\documentclass[12pt]{amsart}

\usepackage{amsmath,amssymb,amsthm}
\usepackage[bookmarksopen=true,
            bookmarksnumbered=true,
            colorlinks=true,
            linkcolor=blue,
            pdftitle={},
            pdfsubject={Commutative Algebra},
            pdfkeywords={regularity, Betti number},
            pdfpagemode=UseOutlines,
            pdfproducer={Latex with hyperref},
            letterpaper=true,
            pdfcreator={latex->dvips->ps2pdf}]{hyperref}
\usepackage{graphicx,epsfig}
\usepackage{enumerate}
\usepackage{xypic}
  \input{xy}
  \xyoption{all}
\usepackage{url}

\numberwithin{equation}{section}

\theoremstyle{plain}
\newtheorem{theorem}{Theorem}[section]

\newtheorem{lemma}[theorem]{Lemma}

\newtheorem{corollary}[theorem]{Corollary}

\newtheorem{proposition}[theorem]{Proposition}

\newtheorem{conjecture}[theorem]{Conjecture}

\theoremstyle{definition}

\newtheorem{remark}[theorem]{Remark}

\def\0{{\bf 0}}

\DeclareMathOperator{\depth}{depth}

\DeclareMathOperator{\fin}{end}

\DeclareMathOperator{\reg}{reg}
\DeclareMathOperator{\tor}{Tor}
\DeclareMathOperator{\Tor}{Tor}

\DeclareMathOperator{\Tot}{Tot}
\DeclareMathOperator{\cha}{char}

\begin{document}

\title{Degree Bounds on Homology and a Conjecture of Derksen}

\author{Marc Chardin}
\address{Institut de Math\'ematiques de Jussieu\\
4, place Jussieu\\
F-75005 Paris\\
France}
\email{chardin@math.jussieu.fr}

\author{Peter Symonds}
\address{School of Mathematics\\
         University of Manchester\\
     Manchester M13 9PL\\
     United Kingdom}
\email{Peter.Symonds@manchester.ac.uk}


\thanks{Second author partially supported by a grant from the Leverhulme Trust.}

\begin{abstract}
Harm Derksen made a conjecture concerning degree bounds for the syzygies of rings of polynomial invariants in the non-modular case \cite{derksen}. We provide counterexamples to this conjecture, but also prove a slightly weakened version.

We also prove some general results that give degree bounds on the homology of complexes and of $\Tor$ groups.
\end{abstract}

\maketitle

\section{Introduction}

Let $G$ be a finite group and $V$ a finite dimensional representation of $G$ over a field $k$. Let $B=k[V]$, graded with $V$ in degree 1, and let $R=B^G$. In characteristic 0, it was shown by Noether that $R$ is generated in degrees at most $|G|$. This is also true whenever $\cha k$ does not divide $|G|$, as was shown more recently by Fleischmann \cite{fleischmann}, Fogarty \cite{fogarty} and Derksen and Sidman \cite{ds}. We refer to this as the non-modular case.

Let $I$ be the ideal in $B$ generated by $R_+$ and let $\tau_G(V)$ be the smallest positive degree $i$ in which $I_i=B_i$, in other words $\tau_G(V)= \fin (B/I) +1$. The proofs mentioned above also show that $\tau_G(V) \leq |G|$ in the non-modular case. In fact, it was shown by Broer \cite[Lemma 6]{broer} that this inequality holds in general provided we assume that the inclusion of $R$ in $B$ is split as a map of $R$-modules.

Let $f_1, \cdots ,f_r$ be a minimal set of generators for $R$ and let $S = k[x_1, \ldots , x_r]$ be a polynomial ring with  $\deg x_i = \deg f_i$. There is a natural surjection $S \twoheadrightarrow R$ given by  $x_i \mapsto f_i$. For any $S$-module $M$, let $t^S_i(M)= \fin \Tor ^S_i(M,k)$; this is equal to the largest degree of a basis element of the $i$th term in the minimal free resolution of $M$ over $S$.

\begin{conjecture}[Derksen \cite{derksen}]
In the non-modular case, $t^S_i(R) \leq (i+1) \tau_G(V) \leq (i+1)|G|$.
\end{conjecture}

Derksen proved the case $i=1$. Recently, Snowden \cite{snowden} showed that $t^S_i(R) \leq i |G|^3$.

In Section~\ref{s:v} we will show that the Veronese subrings form a collection of counterexamples with $\tau_G(V)=|G|$.

However, we do have a positive result that is close to the original conjecture.

\begin{theorem}
\label{t:main}
In the non-modular case, or whenever $R:=B^G \hookrightarrow B$ is split over $B^G$, 
\[
t^S_i(R) \leq (i+1)\tau_G(V) +i-1 \leq (i+1)|G|+i-1.
\]
If, furthermore, $t^S_1(I) \leq \tau_G(V)$, then 
\[
t^S_i(R) \leq (i+1)\tau_G(V) -1 \leq  (i+1)|G|-1.
\]
\end{theorem}

We also have the following result.

\begin{theorem}
\label{t:R}
Suppose that $B^G \hookrightarrow B$ is split over $B^G$ and write $R:=B^G$. Then  $t^R_0(k)=0$, $t^R_1(k) \leq \tau_G(V)$ and $t^R_i(k) \leq \tau_G(V)i+i-2$ for $i \geq 2$.

If $t^B_1(I) \leq \tau_G(V)$ then $t^R_i(k) \leq \tau_G(V)i-1$ for $i \geq 2$.
\end{theorem}

In sections \ref{s:gb} and \ref{s:ss}, we present two different approaches to obtaining general results that give degree bounds on homology groups; these both yield Theorem~\ref{t:main} as a special case. We hope that these results will be of more general interest. Indeed, the subject of degree bounds has been explored by many authors. Particularly relevant here is the work of Bruns, Conca and R{\"o}mer \cite{bcr1, bcr2}, but there is a great deal more interesting work, e.g.\ \cite{aci1, aci2, cm}.

\section{Preliminaries}

We work over a base field $k$. All modules and rings are $\mathbb Z$-graded, ideals and elements are assumed to be homogeneous and homomorphisms preserve degree. All rings are non-negatively graded noetherian $k$-algebras with just $k$ in degree 0. A module is said to be bounded below if it is 0 in large negative degrees. The ideal of elements of positive degree in a ring $A$ is denoted by $A_+$ or $\mathfrak m$.

The end of a module is the largest degree in which it is not 0. This takes the value $\infty$ if the degree is not bounded and $-\infty$ if the module is 0. It has the property that $\fin (M \otimes _k N) = \fin M + \fin N$ always, provided we adopt the rather strange convention that $\infty - \infty = -\infty$. All statements of our results are valid in the generality given if this is understood, although they might be trivial in extreme cases.

Given a ring $A$ and an $A$-module $M$, we define $t^A_i(M):= \fin \Tor ^A_i(M,k)$. If $M$ is bounded below then it has a minimal free resolution and $t^A_i(M)$ is equal to the top degree of a basis element of the $i$th term.

We say that an inclusion of rings $A \hookrightarrow B$ is split if it is split as a homomorphism of $A$-modules. If a finite group acts on $B$ and $A = B^G$ then the inclusion is split by the Reynolds operator $b \mapsto {{1}\over{|G|}} \sum _{g \in G} gb$, provided that $\cha k$ does not divide the order of $G$.

In the modular case, let $P$ denote a Sylow $p$-subgroup of $G$. Then the inclusion $B^G \hookrightarrow B^P$ is split by $b \mapsto {{1}\over{|G/P|}} \sum _{g \in G/P} gb$. It follows that $B^G \hookrightarrow B$ is split if $B^P \hookrightarrow B$ is split. Thus a sufficient condition for $B^G \hookrightarrow B$ to be split is that $B^P$ be polynomial. In fact, it is a conjecture that this condition is necessary for $B^P \hookrightarrow B$ to be split \cite{broer}.

In the context of Derksen's conjecture, we will always assume that the inclusion $R=B^G \hookrightarrow B$ is split. It follows that $t^S_i(R) \leq t^S_i(B)$; all our bounds on $t^S_i(R)$ will be found as bounds on $t^S_i(B)$.

In the construction of the ring $S$ in Derksen's conjecture, it is the bound on the degrees of the generators that is important, not the linear independence of their images in $R$. In this connection, the following lemma will be useful.

\begin{lemma}
\label{l:sg}
Suppose that the ring $B$ is standard graded, $C$ is a subring and the inclusion is split over $C$; set $I:=C_+B$.  Given any set of generators of $C$, the subset consisting of those elements of degree at most $\fin (B/I) +1$ also generates $C$. In particular, a minimal set of generators has elements of degree at most $\fin (B/I) +1$.

If $t^B_1(I) \leq \fin (B/I) +1$, $B$ is an integral domain and $\dim B \geq 2$, then the subset of elements of degree at most $\fin (B/I)$ generates $C$.
\end{lemma}

\begin{proof}
Let $J<B$ be the ideal generated by the elements of the generating set of degree at most $e+1$, where $e = \fin (B/I)$. 
We know that $t^B_0(I)=t^B_1(B/I)\leq \reg (B/I)+1 \leq e +1$ and $I_r=J_r$ for $r \leq e+1$; thus $I=J$. Applying the splitting map from $B$ to $C$ yields the result. 

For the second part, note that we must have $t^B_0(I) \leq e$. This is because it is well known and easy to show that if $t^B_1(I) \leq  t^B_0(I)$ and $B$ is an integral domain then $I$ must be principal; thus $\fin (B/I) < \infty$ entails $\dim B \leq 1$. The rest of the proof proceeds as before. 
\end{proof}

\section{Veronese Subrings}
\label{s:v}

Let $Q^n=k[x_1, \ldots ,x_n]$, with the $x_i$ in degree 1. Given $m \in \mathbb N$, the $m$th Veronese subring is $V^{n,m} = \oplus _i Q^n_{im}$. 

Let $G$ be a cyclic group of order $m$ with generator $g$. Suppose that $w \in k$ is a primitive $m$th root of unity and let $G$ act on $Q$ by $gx_i=wx_i$. Then $(Q^n)^G=V^{n,m}$. Thus, as long as $k$ contains a primitive $m$th root of unity, $V^{n,m}$ is a ring of invariants, and the map $V^{n,m} \hookrightarrow Q^n$ is split, for degree reasons.

Let $P^{n,m}=k[x^m_1, \ldots , x^m_n ]$ and $S^{n,m}=k[Q^n_m]$. Then there is a natural surjection $S^{n,m} \twoheadrightarrow V^{n,m}$ and $Q^n$ is free over $P^{n,m}$ with basis the monomials $x^{i_1}_1,  \cdots , x^{i_n}_n$, $0 \leq i_j \leq m-1$. Thus $V^{n,m}$ is free over $P^{n,m}$ with basis the subset of these monomials with $i_1 + \cdots + i_n$ divisible by $m$. The highest degree of such a monomial is $nm-\lceil \frac{n}{m} \rceil m$. Now let $\bar{V}^{n,m}$, $\bar{P}^{n,m}$ and $\bar{S}^{n,m}$ be $V^{n,m}$, $P^{n,m}$ and $S^{n,m}$ with the degrees divided by $m$; thus $\bar{P}^{n,m}$ is standard graded and $\bar{V}^{n,m}$ is free over $\bar{P}^{n,m}$ with the highest degree of a basis element being $n-\lceil \frac{n}{m} \rceil $.

It follows that $\reg \bar{V}^{n,m} = n-\lceil \frac{n}{m} \rceil $. But, working over $\bar{S}^{n,m}$, $\reg \bar{V}^{n,m} = \max _i \{ t^{\bar{S}^{n,m}}_i (\bar{V}^{n,m}) -i \}$. It follows that there is an $i$ such that $ t^{\bar{S}^{n,m}}_i(\bar{V}^{n,m})-i=n-\lceil \frac{n}{m} \rceil $. Multiplying all degrees by $m$, we see that $ t^{S^{n,m}}_i(V^{n,m})-im=nm-\lceil \frac{n}{m} \rceil m$. 

Derksen's conjecture predicts that $ t^{S^{n,m}}_i(V^{n,m}) \leq (i+1)m$, so there is a contradiction if $n-\lceil \frac{n}{m} \rceil >1$. This happens if $n=3$ and $m \geq 3$ or $n \geq 4$ and $m \ne 1$.

That Veronese subrings give counterexamples can also be deduced from \cite[Corollary 4.2]{bcr1}. 

\section{General Bounds for Complexes}
\label{s:gb}

The next lemma is a standard consequence of local duality when $B$ is a polynomial ring. 

\begin{lemma}
\label{l:boundtor}
Let $B$ be a $k$-algebra and let $M$ and $N$ be two $B$-modules such that $N$ is bounded below. Suppose that $\dim \Tor ^B_i(M,N) \leq 1$ for $i \geq 1$. Then
\[
\max \{ \fin H^0_{\mathfrak m} (\Tor ^B_i(M,N)) , \fin H^1_{\mathfrak m} (\Tor ^B_{i+1}(M,N)) \} \leq \max_{0 \leq j \leq \dim M}\{\fin H_{ \mathfrak m}^j(M) +  t^B_{i+j}(N) \}.
\]
\end{lemma}

\begin{proof}
Let $C$ be the \v{C}ech complex on a homogeneous system of parameters of $B$ and let $F$ be a minimal free resolution of $N$ (this is where we require $N$ to be bounded below). Consider the double complex $Y=C \otimes F \otimes M$, where $Y_{p,q}=C^{-p} \otimes _B F_q \otimes _B M$, and its associated spectral sequences.

We have ${}^IE^1_{p,q} \cong \Tor ^B_q(N,C^{-p} \otimes _B M) \cong \Tor ^B_q(N,M) \otimes _B C^{-p}$, since $C^{-p}$ is flat; thus ${}^IE^2_{p,q} \cong H^{-p}_{ \mathfrak m}(\Tor^B_q(N,M))$. The hypothesis that $\dim \Tor ^B_i(M,N) \leq 1$ for $i \geq 1$ implies that ${}^IE^2_{p,q}=0$ if $q \ne 0$ and $p \ne 0,-1$. The spectral sequence collapses and we obtain $H_i(\Tot Y) \cong H^0_{\mathfrak m} \Tor ^B_i(N,M) \oplus H^1_{\mathfrak m} \Tor ^B_{i+1}(N,M)$ as $k$-modules for $i \geq -1$.

But ${}^{II}E^1_{p,q} \cong H^{-q}_{ \mathfrak m}(F_p \otimes _B M) \cong H^{-q}_{ \mathfrak m}( M) \otimes_B F_p$, since $F_p$ is flat. 
Because $F_\bullet$ is minimal, we know that the top degree of a basis element of $F_p$ is $t^B_p (N)$; hence $\fin (H^{-q}_{ \mathfrak m}( M) \otimes_B F_p) \leq \fin H^{-q}_{ \mathfrak m}( M) +  t^B_p (N)$.

Thus $\fin H_i(\Tot Y) \leq \max _{p+q=i} \{ \fin H^{-q}_{ \mathfrak m}( M) +  t^B_p (N) \}$.
\end{proof}

Given a complex $L$, we set $H_i:=H_i(L)$ and $Z_i:=\ker ( d_i \! : L_i \rightarrow L_{i-1})$.  

\begin{proposition}
\label{p:Lind}
Let $L$ be a complex of $B$-modules and set $c^j_i(L)= \max_{k \geq 0} \{ \fin H^{j-k}_{\mathfrak m}(L_{i-k}) \}$ and $\dim B =n$. Then, for any $i \in \mathbb Z$,
\[
\reg Z_i \leq \max_{2-k \leq j \leq n } \{ \fin H^k_{\mathfrak m} (H_{i+1-j-k})+j, \enskip c^j_i(L)+j \}.
\]
If $I<B$ is an ideal such that $\dim B/I \leq 1$ and $IH_i=0$, then
\begin{align*}
\fin H^0_{\mathfrak m}(H_i) & \leq \max \{ \{ \max \{ \fin H^k_{\mathfrak m}(H_{i+1-j-k}), \enskip  c^j_i(L) \} + t^B_j(B/I) \}_{2 \leq j+ k }, \enskip \fin H^1_{\mathfrak m}(L_{i+1} \otimes B/I) \} \\
\fin H^1_{\mathfrak m}(H_i) & \leq  \max_{2 \leq j+k } \{ \max \{ \fin H^k_{\mathfrak m}(H_{i+1-j-k}), \enskip  c^j_i(L) \} + t^B_{j-1}(B/I) \}.
\end{align*}
In particular, if $H_j$ is ${\mathfrak m}$-torsion for $i-n+1 \leq j \leq i-1$ and $\depth L_{i-k} \geq \min \{ n+1-k,n \}$ for $k \geq 0$, then
\[
\reg Z_i \leq \max \{ \{ \fin (H_{i-j+1} )+j \}_{2 \leq j \leq n} , \enskip \fin H^n_{\mathfrak m} (L_i) +n \}.
\]
If, in addition, $\dim B/I=0$ and $IH_i=0$, then
\[ 
\fin H_i \leq \max \{ \{ \fin (H_{i-j+1}) + t^B_j(B/I) \}_{2 \leq j \leq n}, \enskip \fin H^n_{\mathfrak m} (L_i) + t^B_n(B/I) \}.
\]
\end{proposition}

If the modules are not finitely generated then their depth is defined in terms of local cohomology.

\begin{proof}
Let $L^{(i)}$ be the truncated complex $0 \rightarrow Z_i \rightarrow L_i \rightarrow L_{i-1} \rightarrow  \cdots$. Note that $H_j(L^{(i)})$ is equal to $H_j$ if $j \leq i-1$ and 0 otherwise. 

Let $C$ be the \v{C}ech complex on a homogeneous system of parameters of $B$. Consider the double complex $X=C \otimes L^{(i)}$, where $X_{p,q}=C^{-q} \otimes _B L^{(i)}_p$, and its associated spectral sequences.

We have ${}^IE^1_{p,q} \cong C^{-p} \otimes _B H_q(L^{(i)})$, since $C^{-p}$ is flat; hence ${}^IE^2_{p,q} \cong H^{-p}_{ \mathfrak m} (H_q)$ if $q \leq i-1$ and is zero otherwise. 
Thus $\fin H_j (\Tot X) \leq \max_{q-p=j, \, q \leq i-1} \{ \fin H^{-p}_{\mathfrak m}(H_q) \}$. 

Also ${}^{II}E^1_{p,q} \cong H^{-q}_{ \mathfrak m}(L^{(i)}_p) \cong \begin{cases} H^{-q}_{ \mathfrak m} (L_p) & \mbox{ if $p \leq i$} \\ H^{-q}_{\mathfrak m}(Z_i) & \mbox{ if $p = i+1$} \\ 0 & \mbox{ otherwise.} \end{cases}$

Now $\fin {}^{II}E^1_{p,q}$ is bounded by the largest of $\fin {}^{II}E^\infty_{p,q}$ and the ends of the terms to which it is connected by a differential on some page. There is no incoming differential reaching a module $^{II}E^r_{i+1,q}$, $r \geq 1$, and the outgoing one lands in ${}^{II}E^r_{i-r+1,q+r-1}$, where $\fin {}^{II}E^r_{i-r+1,q+r-1} \leq \fin {}^{II}E^1_{i-r+1,q+r-1} = \fin H^{q+r-1}_{\mathfrak m}(L_{i-r+1}) \leq c^q_i(L)$. Since $\fin H^j_{\mathfrak m}(Z_i)= \fin {}^{II}E^1_{i+1,-j}$, it is bounded by the larger of $\fin H_{i+1-j} ( \Tot X)$ and $c^j_i(L)$. 

Putting this together and using the hypothesis $\dim B/I \leq 1$, we obtain $\fin H^j_{\mathfrak m} (Z_i) \leq \max \{ \{ \fin H^{-p}_{\mathfrak m} (H_{i+1-j+p}) \}_{ p \leq j-2 }, \enskip c^j_i(L) \}$. 

The first formula for $\reg Z_i$ follows immediately from the definition of regularity.

Because $IH_i=0$, multiplication gives a surjection $Z_i \otimes _B B/I \twoheadrightarrow H_i$. If $\dim B/I =0$ it follows that $\fin H_i \leq \fin Z_i \otimes _B B/I $. Otherwise, there is a an exact sequence of $B/I$-modules $0 \rightarrow K \rightarrow L_{i+1} \otimes B/I \rightarrow Z_i \otimes B/I \rightarrow H_i \rightarrow 0$ for some module $K$. Splitting this into two short exact sequences and using the long exact sequence in local cohomology and the hypothesis that $\dim B/I \leq 1$ yields $\fin H^0_{\mathfrak m}(H_i) \leq \max \{ \fin H^0(Z_i \otimes B/I) , \enskip \fin H^1_{\mathfrak m} (L_{i+1} \otimes B/I) \}$ and $\fin H^1_{\mathfrak m}(H_i) \leq \fin H^1_{\mathfrak m} (Z_i \otimes B/I)$.    The local cohomology of the tensor product can be estimated using Lemma~\ref{l:boundtor}, which yields the next pair of formulas. 

The last part follows because the conditions on $\depth L_{i-k}$ force $c^j_i(L)=- \infty$ for $j \leq n-1$ and $c^n_i(L)= \fin H^n_{\mathfrak m}(L_i)$.
\end{proof} 

For $I<B$, set
$$
T_{j}(I):=\max\{ t^B_{i_1}(I)+\cdots +t^B_{i_r}(I)\ \vert\ i_1>0,\ldots ,i_r>0, i_1+\cdots +i_r=j\}.
$$ 
when $j>0$ and $T_{j}(I):= - \infty$ for $j \leq 0$. 
 
\begin{theorem}
\label{t:L}
Let $L$ be a complex of $B$-modules and $i$ an integer such that the $L_j$ for $j \leq i$ have depth $n = \dim B$ and are bounded below. Suppose that $I<B$ is an ideal such that $\dim B/I =0$, $IH_j=0$ for $j \leq i$ and that $H_j=0$ for $j<<0$. Then,
\[
\fin H_i \leq \max _{j<i} \{ \fin H^n_{\mathfrak m} (L_j) + T_{i-j}(I) \} + t^B_n(B/I).
\]
\end{theorem}

\begin{proof}
This is a straightforward induction on $i$ using Proposition~\ref{p:Lind}. We start at some $i$ such that $H_j=0$ for $j \leq i$ and use the fact that, by design, $T_{j+k}(I) \geq T_j(I)+t^B_{k+1}(B/I)$.
\end{proof}

\begin{lemma}
\label{l:overb}
Suppose that $B$ is a standard graded polynomial ring and $I<B$. Then
\begin{enumerate}
\item
$T_i(I) \leq (\fin B/I +2)i$ and
\item
if $t^B_1(I) \leq \fin B/I +1$ then $T_i(I) \leq (\fin B/I +1)i$. 
\end{enumerate}
\end{lemma}

\begin{proof}
Calculating $\Tor^B_i(B/I,k)$ using the standard Koszul complex for $B$ shows that $t^B_i(B/I) \leq i + \fin B/I$, so $t^B_i(I) \leq i + 1 + \fin B/I$. The result now follows from an easy induction on $i$.
\end{proof}

\begin{corollary}
\label{c:koszul}

Let $B$ be a standard graded polynomial ring and $I<B$ an ideal such that $\fin B/I < \infty$. Let $f_1, \ldots , f_m$ be a set of generators for $I$ that are minimal by degree and let $K(f;B)$ be the Koszul complex on the $f$. Then
\[
\fin H_i(f;B) \leq (\fin B/I +2)(i+1)-2.
\]
If $t^B_1(I) \leq \fin B/I +1$ then
\[
\fin H_i(f;B) \leq (\fin B/I +1)(i+1)-1.
\]
\end{corollary}

\begin{proof}
Let $d$ be the maximum of the degrees of the $f_i$; Lemma~\ref{l:sg} shows that $d \leq \fin B/I +1$.

Now $\fin H^n_I(K_i) \leq di -n \leq (\fin B/I +1) -n$ for $0 \leq i \leq m$ and is $-\infty$ otherwise. The result follows from Theorem~\ref{t:L} and Lemma~\ref{l:overb}.
\end{proof}

\begin{remark}
The first part of this corollary appears in \cite[Proposition 3.3]{bcr2}, with an elementary proof.
\end{remark}

In the context of Derksen's conjecture, we can calculate $\Tor^S_i(B,k)$ by using the  Koszul complex $K(x;S)$ to resolve $k$ over $S$, then tensoring with $B$ to obtain $K(f;B)$. Thus $t^S_i(B) = \fin H_i(f;B)$, and Theorem~\ref{t:main} follows from Corollary~\ref{c:koszul}.

\section{The Change of Rings Spectral Sequence} 
\label{s:ss}

Let $f \! : A \rightarrow B $ be a homomorphism of $k$-algebras and let $I=f(A_+)B<B$.

The change of rings spectral sequence $E^2_{p,q}=\tor_p^B(\tor_q^A(B,k),k)\Rightarrow \tor_{p+q}^A(k,k)$ has the following form.
$$
\footnotesize
\xymatrix{
k \otimes _B \tor_2^A(B,k)&&&\\
k \otimes _B \tor_1^A(B,k)&\tor_1^B(\tor_1^A(B,k),k)&\tor_2^B(\tor_1^A(B,k),k)\ar[ull]&&\\
B \otimes _A k&\tor_1^B(B \otimes _A k,k)&\tor_2^B(B \otimes _A k,k)\ar[ull]&\tor_3^B(B \otimes _A k,k)\ar[ull]\ar@{-->}[uulll]& \tor_4^B(B \otimes _A k,k)\ar[ull]\\}
$$
Filtering the $B$ in $\tor_q^A(B,k)$ by degree, we see that $\fin \tor_p^B(\tor_q^A(B,k),k)\leq \fin \tor_q^A(B,k)+ \fin \tor_p^B(k,k)$.

The end of an entry on the $E_2$ page is bounded by the largest of the end of $H_i(\Tot)$ corresponding to its diagonal and the ends of the $E_2$ entries that are linked to it by a differential on some page. Applying this to the bottom row yields
\begin{equation}
\label{e:ss1}
\fin \Tor ^B_i(B \otimes _A k,k) \leq \max \{ \{ t^B_j(k) + t^A_{i-j-1}(B) \}_{0 \leq j \leq i-2},  t^A_i(k) \}.
\end{equation}
From the first column we obtain
\begin{equation}
\label{e:ss2}
\fin (k \otimes _B \Tor ^A_i(B,k)) \leq \max \{ \{ t^A_j(B) + t^B_{i-j+1}(k) \}_{0 \leq j \leq i-1},  t^A_i(k) \}.
\end{equation}
Notice that $\Tor^A_i(B,k)$ is naturally a $B/I$-module and as such is generated in degrees at most $\fin (k \otimes _B \Tor ^A_i (B,k))$. Thus
\begin{equation}
\label{e:ss3}
\fin \Tor^A_i (B,k) \leq \fin (k \otimes _B \Tor ^A_i(B,k)) + \fin B/I.
\end{equation}
Set
$$
U_{j}(f):=\max\{ (t^B_{i_1}(B_+) + \fin B/I)+\cdots +(t^B_{i_r}(B_+) + \fin B/I) \ \vert \ i_1>0,\ldots ,i_r>0, i_1+\cdots +i_r=j \}
$$ 
when $j>0$ and $U_{j}(f):= - \infty$ for $j \leq 0$. 

\begin{proposition}
\label{p:ssbound}
We have
\[
t^A_i(B) \leq \max \{ U_i(f), \enskip \{ t^A_j(k) + (i-j) \fin B/I \}_{0 \leq j \leq i} \} + \fin B/I.
\]
\end{proposition}

\begin{proof}
By design, $U_{j+k}(f) \geq U_{j}(f) + t^B_{k+1}(k) + \fin B/I$. The proof is now by induction on $i$, using inequalities \ref{e:ss2} and \ref{e:ss3}.
\end{proof}

If $B$ is a polynomial ring with generators in degrees $e_1 \geq e_2 \geq \cdots \geq e_n$, then $U_{i}(f)=(e_1+e_2 + \fin B/I)i$ for $n\geq 2$ and is $-\infty$ otherwise. For convenience, we set $e_2= -\infty$ when $n \leq 1$. When $B$ is standard graded and $n \geq 2$ this becomes $U_{i}(f)=(2 + \fin B/I)i$.

If $A$ is a polynomial ring with generators in degrees $d_1 \geq d_2 \geq \cdots \geq d_m$, we set $\overline{d}_j= \max \{ d_j, \enskip \fin B/I \}$ for $j \leq m$, $\overline{d}_j= \fin B/I$ for $j > m$. Then $\max_{0 \leq j \leq i} \{ t^A_j(k) + (i-j) \fin B/I \} =  \overline{d}_1+ \cdots + \overline{d}_i$.

By Lemma~\ref{l:sg}, there is always a subset of the generators of $A$ that still generates $I$ and satisfies $\overline{d}_j \leq \fin B/I +1$ for all $j$. We call such a set small.

\begin{corollary}
\label{c:ssbound}
When both $A$ and $B$ are polynomial rings,
\[
t^A_i(B) \leq \{ (e_1+e_2+ \fin B/I)i, \enskip \overline{d}_1+ \cdots + \overline{d}_i \} + \fin B/I.
\]
If $B$ is standard graded and the generators of $A$ map to a small set of generators for $I$, then
\[
t^A_i(B) \leq (\fin B/I +2)(i+1)-2.
\]
\end{corollary}

\begin{proof} 
This is immediate from the preceding remarks.
\end{proof}

Let $C= f(A) \subseteq B$ . Suppose that the inclusion $C \hookrightarrow B$ is split as a map of $C$-modules. Then the bounds that we have obtained for $t^A(B)$ are also valid for $t^A(C)$. The first part of Theorem~\ref{t:main} follows.

\begin{lemma}
\label{l:tor2}
Let $X$ be a $B/I$-module that is bounded below. Then 
\[
t^B_i(X) \leq t^B_0(X) + \max \{ t^B_i(B/I), \enskip \fin B/I + t^B_{i-1}(k) \}.
\]
\end{lemma}

\begin{proof}
Express $X$ as the quotient of a minimal free $B/I$-module $F$ and let $Y$ be the kernel, so we have a short exact sequence $0 \rightarrow Y \rightarrow F \rightarrow X \rightarrow0$ and $t^B_0(F) \cong t^B_0(X)$. We also have $\fin Y \leq \fin F = t^B_0(F) + \fin B/I$.

Part of the long exact sequence for $\Tor^B(-,k)$ is
\[
\cdots \rightarrow \Tor^B_i(B/I,k) \rightarrow \Tor^B_i(X,k) \rightarrow \Tor^B_{i-1}(Y,k) \rightarrow \cdots .
\] 
We also know that $\fin \Tor^B_i(F,k) = t^B_0(F) + t^B_i(B/I)$ and $\fin \Tor^B_{i-1}(Y,k) \leq t^B_0(Y) + t^B_{i-1}(k)$. Putting all this together, we obtain the result.
\end{proof}

\begin{corollary}
\label{c:ssboundplus}
Let $A$ and $B$ be polynomial rings such that $B$ is standard graded and the generators of $A$ map to a small set of generators for $I$. If $t^B_1(I) \leq \fin B/I+1$, then
\[
t^A_i(B) \leq (\fin B/I +1)(i+1)-1.
\]
\end{corollary}

\begin{proof}
In view of inequality \ref{e:ss3}, it is sufficient to prove that $t^A_0( \Tor^A_i(B,k)) \leq (\fin B/I + 1)i$. We do this by induction on $i$.

Using Lemma~\ref{l:tor2} with $X=\Tor^A_{i-1}(B,k)$, we obtain $t^B_2(\Tor^A_{i-1}(B,k) )\leq t^A_0( \Tor^A_{i-1}(B,k)) + \max \{ t^B_2(B/I), \enskip t^B_1(k) + \fin B/I \}$. But $t^A_0(\Tor^A_{i-1}(B,k)) \leq (\fin B/I +1)i$, by induction, $t^B_2(B/I)=t^B_1(I) \leq \fin B/I +1$ and $t^B_1(k)=1$. Thus $t^B_2(\Tor^A_{i-1}(B,k)) \leq (\fin B/I +1)i$.

We can use this estimate instead of $t^B_2(k) + t^A_{i-1}(B)$ in inequality \ref{e:ss2}. For the other terms we have $t^A_j(B) \leq (\fin B/I +1)(j+1)-1$, by induction, and $t^B_{i-j+1}(k) =i-j+1$. This leads to the required result.
\end{proof}

The second part of Theorem~\ref{t:main} follows. 

We can also obtain bounds on $t^C(k)$.

\begin{theorem}
Suppose that $B$ is a standard graded polynomial ring, $C$ is a subring and the inclusion is split over $C$. Let $I:= C_+B$.  Then $t^C_0(k)=0$, $t^C_1(k) \leq \fin B/I+1$ and $t^C_i(k) \leq (\fin B/I +2)i-2$ for $i \geq 2$.

If $t^B_1(I) \leq \fin B/I +1$, then $t^C_i(k) \leq (\fin B/I +1)i-1$ for $i \geq 2$.
\end{theorem}

\begin{proof}
By Lemma~\ref{l:sg}, there is a set of generators for $C$ in degrees at most $\fin (B/I)+1$. Using these we form a polynomial ring $A$ that maps onto $C$ as in the statement of Derksen's Conjecture.

Since $k \cong C \otimes _A k$, we have $t^C_i(k) = \fin \Tor^C_i(C \otimes _A k,k)$. 

We can bound the latter by applying inequality \ref{e:ss1} to the map $A \twoheadrightarrow C$ to obtain $
\fin t^C_i(k) \leq \max \{ \{ t^C_j(k) + t^A_{i-j-1}(C) \}_{0 \leq j \leq i-2},  t^A_i(k) \}$.

But $t^A_i(k) \leq (\fin B/I + 1)i$, because we can calculate this using the Koszul complex and the degrees of the generators are bounded.
Corollary~\ref{c:ssbound} and the fact that $C \hookrightarrow $B is split yield $t^A_{i-j-1}(C) \leq (\fin B/I +2)(i-j)-2$ or $t^A_{i-j-1}(C) \leq (\fin B/I +1)(i-j)-1$.
The rest of the proof is a straightforward induction that is left to the reader.

When $t^B_1(I) \leq \fin B/I +1$, we can improve the bound on the degrees of the generators by 1 using the second part of Lemma~\ref{l:sg} (our result is vacuous for $\dim B \leq 1$).
Corollary~\ref{c:ssboundplus} yields
$t^A_{i-j-1}(C) \leq (\fin B/I +1)(i-j)-1$ and again we finish by induction.
\end{proof}

Theorem~\ref{t:R} follows.

\end{document}